\documentclass[10pt,a4paper,reqno,twoside]{amsart}

\usepackage{verbatim}
\usepackage{amssymb}
\usepackage{upref}
\usepackage{enumerate}
\usepackage{color}
\usepackage{latexsym, amsfonts, amsmath, amscd, esint}
\usepackage{mathrsfs}

\makeatletter

\theoremstyle{plain}
\newtheorem{Thm}{Theorem}[section]

\newtheorem{lemma}{Lemma}[section]
\newtheorem{corollary}{Corollary}[section]
\theoremstyle{definition}

\newtheorem*{Not}{Notation}
\theoremstyle{remark}
\newtheorem{Rem}{Remark}[section]
\newtheorem{Example}[Rem]{Example}

\allowdisplaybreaks
\newcommand{\N}{\mathbb{N}}

\newcommand{\R}{\mathbb{R}}

\newcommand{\xii}{{\abs{\xi}}}

\DeclareMathOperator{\supp}{supp}

\def\<#1\>{\left\langle#1\right\rangle }

\newcommand{\abs}[1]{\left\vert#1\right\vert}

\renewcommand{\doteq}{{\mathrm{\,:=\,}}}

\title[evolutions  equations with structural damping]{The influence of data regularity in the critical exponent \\
for a class of semilinear evolutions  equations }

\author {Marcelo R. Ebert}
\address{Departamento de Computa\c{c}\~ao e Matem\'atica, Universidade de S\~ao Paulo, Ribeir\~ao Preto, SP, 14040-901, Brasil\\
email: ebert@ffclrp.usp.br}

\author {Cleverson R. da Luz}
\address{Department of Mathematics, Federal University of Santa Catarina, Campus Trindade, Florian\'opolis, SC, 88040-900, Brazil\\
	email: cleverson.luz@ufsc.br}

\author {Ma\'ira F. G. Palma}
\address{Department of Mathematics, Federal University of Santa Catarina, Campus Trindade, Florian\'opolis, SC, 88040-900, Brazil\\
	email: mairagauer@gmail.com}

\begin{document}

\begin{abstract}

In this paper we find the critical exponent for the global existence (in time) of small data solutions to the Cauchy problem
for the semilinear dissipative evolution  equations
\[
u_{tt}+(-\Delta)^\delta u_{tt}+(-\Delta)^\alpha
u+(-\Delta)^\theta u_t=|u_t|^p, \quad t\geq 0,\,\, x\in\R^n,\]
with $p>1$,    $2\theta \in [0,  \alpha]$ and $\delta \in (\theta,\alpha]$. We show that, under additional regularity  $\left(H^{\alpha+\delta}(\R^n)\cap L^{m}(\R^n) \right)\times  \left(H^{2\delta}(\R^n)\cap L^{m}(\R^n)\right) $ for initial data, with $m\in (1,2]$, the critical exponent is given by $p_c=1+\frac{2m\theta}{n}$. The nonexistence of global solutions in the  subcritical cases is proved, in the case of integers
parameters $\alpha, \delta, \theta$,
by using the test function method (under suitable sign assumptions on the initial data).

\end{abstract}



\maketitle

\section{Introduction}\label{sec:intro}

\baselineskip18pt

Let us consider  the Cauchy problem  for the semilinear dissipative evolution  equations
\begin{equation}\label{eq:CP0}
\begin{cases}
u_{tt}+(-\Delta)^\delta u_{tt}+(-\Delta)^\alpha
u+(-\Delta)^\theta u_t=|u_t|^p, \quad t\geq 0, \,\,x\in\R^n,\\
(u,u_t)(0,x)=(u_0,u_1)(x),
\end{cases}
\end{equation}
with $p>1$, $2\theta\in [0, \alpha]$ and  $\delta \in [0,\alpha]$.
Here we denote by $(-\Delta)^{\frac{b}2}=|D|^b$, with $b \geq 0$, the fractional Laplacian operator defined by its action~$|D|^bf=\mathfrak F^{-1}(\xii^b\hat f)$, where $\mathfrak F$ is the Fourier transform with respect to the space variable, and~$\hat f=\mathfrak F f$.
The case $\alpha = 2$ and $\delta=0$ in \eqref{eq:CP0} is an important model in the literature, it is known as
Germain-Lagrange operator, as well as beam operator and plate operator in the case
of space dimension $n = 1$ and $n = 2$, respectively.

Models to study the
vibrations of thin plates  given by the full von K\'arm\'an system have been studied by several authors,
in particular, see \cite{C}, \cite{MZ}. If $\delta=1$ in \eqref{eq:CP0}, the term $-\Delta u_{tt}$  is to absorb in the system the rotational
inertia effects at the point $x$ of the plate in a positive time $t$. For the plate equation with exterior damping
\begin{equation}\label{plateeq:CPlinear}
\begin{cases}
u_{tt}-\Delta u_{tt}+(-\Delta)^2
u+ u_t=f(u, u_t), \quad t\geq 0, \,\,x\in\R^n,\\
(u,u_t)(0,x)=(u_0,u_1)(x),
\end{cases}
\end{equation}
 where $f(u, u_t)=|\partial_t^j u|^p, j=0,1$, we address the reader to \cite{CL}, \cite{DA17}, \cite{L} and \cite{SK}  for a detailed investigation of properties like existence,  uniqueness,  energy estimates for the solution  and  global existence (in time) of small data solutions.
 The derived estimates in Section \ref{sec:linear} for solutions to the associate linear problem to \eqref{eq:CP0}   could also be applied to generalize  the obtained results in \cite{DA17}, namely,  problem \eqref{plateeq:CPlinear} with power nonlinearity $|u|^p$ and, under additional regularity $L^{m}(\R^n)$, one may expect that the critical exponent  for the global existence (in time) of small data solutions is $\bar p=1+\frac{2m\alpha}{(n-2m\theta)_+}$. But due to fact that only partial results are obtained  in the literature for \eqref{plateeq:CPlinear} with $f(u, u_t)=|u_t|^p$, in this paper we restrict ourselves to the last power nonlinearities.

It is worth to recall some well known result for dissipative  evolutions models without the rotational
inertia term and with power nonlinearities $|u|^p$. For the classical semilinear damped wave equation
\begin{equation}\label{eqwave}
 u_{tt}-\Delta u + u_t = f(u), \qquad u(0,x)=u_0(x), \qquad u_t(0,x)= u_1(x),
 \end{equation}
with $f(u)=|u|^p$, it was proved in~\cite{TY} the global existence of small data energy solutions
in the \textit{supercritical} range~$p>1+\frac{2}{n}$, by assuming  compactly
supported small data from the energy space. Under  additional regularity  the compact support assumption on the data can be removed. By  assuming data in Sobolev spaces with additional regularity $L^1(\R^n)$, a global (in time) existence result was proved in space dimensions~$n=1,2$ in~\cite{IMN04}, by using energy methods, and in space dimension~$n\leq5$ in~\cite{N04}, by using~$L^r-L^q$ estimates, $1\leq r\leq q\leq \infty$.
Nonexistence of weak global (in time) small data solutions is proved in~\cite{TY}  in the \textit{subcritical} case ~$1<p< 1+\frac{2}{n}$ and in \cite{Z} for $p= 1+\frac{2}{n}$. The exponent~$1+\frac{2}{n}$ is well known as Fujita exponent and it is the critical power for the  semilinear parabolic Cauchy problem (see~\cite{F66}). If one removes the assumption that the initial data are in $L^1(\R^n)$ and we only assume that they are in the energy space, then \cite{IO02} the critical exponent is modified to $1+\frac{4}{n}$ or to $1+\frac{2m}{n}$ under  additional regularity  $L^m(\R^n)$, with $m\in (1,2]$.  In \cite{IIOW} the authors proved that,
differently from the case $m=1$,  the critical exponent $p=1+\frac{2m}{n}$, with $m\in [m_0, 2]$, for some $m_0>1$, belongs to the supercritical case.
More recently, under  additional regularity $L^1(\R^n)$ for initial data and with  $f(u)=|u|^{1+ \frac2{n}}\mu(|u|)$,  in \cite{EGR} the authors obtained sharp conditions on the modulus of continuity function $\mu$ in order to determine a threshold between global (in time) existence of small data solutions (stability of the zero solution) and blow-up behavior even of small data solutions to problem \eqref{eqwave}.

Now, let us consider the Cauchy problem for the dissipative evolution equation
\begin{equation}\label{eqev:CP0}
\begin{cases}
u_{tt}+(-\Delta)^\alpha
u+(-\Delta)^\theta u_t=|\partial_t^j u|^{p}, \quad t\geq 0, \,\,x\in\R^n,\\
(u,u_t)(0,x)=(u_0,u_1)(x),
\end{cases}
\end{equation}
with $j=0,1$.  The term $(-\Delta)^\theta u_t$ represents a damping term.
If $\theta>0$ in \eqref{eqev:CP0}, the damping is said to be structural. The assumption $2\theta \leq \alpha$ means that the damping is \emph{effective}, according to the classification introduced in~\cite{DAE15}. The hypothesis of the damping be effective is in order that the multipliers associated to the linearized  problem has no oscillations at low frequencies in the phase space. In this case,  the asymptotic profile of  solutions   may be described by the solution to an anomalous diffusion problem~\cite{K00}.
In \cite{DAE17NA}, under additional regularity $L^1(\R^n)$ for the initial data, the authors proved that    the critical exponent $p_j, j=0,1$, for global small data solutions to \eqref{eqev:CP0} are, respectively, $p_0
     \doteq 1 + \frac{2\alpha}{n-2\theta}, n> 2\theta $ and
$p_1\doteq 1+\frac{2\theta}{n}$.

Having in mind that the asymptotic profile of solutions to the linear part of the equation influences the
critical exponent for the problem with power nonlinearity, we consider the linear evolution equation related
to \eqref{eq:CP0}:
\begin{equation}\label{eq:CPlinear}
\begin{cases}
u_{tt}+(-\Delta)^\delta u_{tt}+(-\Delta)^\alpha u+(-\Delta)^\theta u_t=0, \quad t\geq 0, \,\,x\in\R^n,\\
(u,u_t)(0,x)=(u_0,u_1)(x).
\end{cases}
\end{equation}
The total energy for \eqref{eq:CPlinear} is
\[ E(t)= \frac12 \| u_t(t, \cdot)\|_{L^2}^2 + \frac12 \|(-\Delta)^{\frac{\delta}{2}} u_t(t, \cdot)\|_{L^2}^2 +  \frac12 \|(-\Delta)^{\frac{\alpha}{2}} u(t, \cdot)\|_{L^2}^2,\]
 and it dissipates,  i.e.,
 \[E'(t)=-\| (-\Delta)^{\frac{\theta}{2}}u_t(t, \cdot)\|_{L^2}^2,\]
 so that, a natural  space  for solutions is $\mathcal C([0,\infty), H^{\alpha}(\R^n)) \cap C^1([0,\infty), H^{\delta}(\R^n))$, the so-called energy solution space.

If $\delta\leq \theta$, the presence of the structural damping  generates a strong smoothing effect on the solution to~\eqref{eq:CPlinear}, and it guarantees the exponential decay in time of the high-frequencies part of the solution to~\eqref{eq:CPlinear}. Therefore, the decay rate for~\eqref{eq:CPlinear} is only determined by the low-frequencies part of the solution to~\eqref{eq:CPlinear}, which behaves like the solution to the corresponding anomalous diffusion problem~\cite{K00}.
However, if $\delta> \theta$, the rotational inertia term $(-\Delta)^\delta u_{tt}$  creates a structure of regularity-loss type decay in the linear problem
(see Theorem \ref{teo>} in Section \ref{sec:linear}) and it is more difficult to apply these linear estimates to study semilinear problems. This fact can be observed by analysing the structure of the eigenvalues associated with the problem \eqref{eq:CPlinear} in the Fourier space. Due to that special structure, when we get estimates for solutions in the region of high frequencies it is necessary to impose additional regularity on the initial data to obtain the same decay estimates as in the region of low frequencies. Such decay property of the regularity-loss type was also investigated for the dissipative Timoshenko system \cite{KISK}, the plate equation under rotational inertia effects in $\mathbb{R}^n$ \cite{CLI, SK} and a hyperbolic-elliptic system of a radiating gas model \cite{TKSK}.

In this work we are interested in the problem \eqref{eq:CP0} with the property of regularity-loss and effective damping, i.e., $\delta>\theta$ and $\alpha \geq 2 \theta$. Our main goal in this paper is to show that, under additional regularity  $\left(H^{\alpha+\delta}(\R^n)\cap L^{m}(\R^n) \right)\times  \left(H^{2\delta}(\R^n)\cap L^{m}(\R^n)\right) $ for initial data, with $m\in (m_0,2]$ and $m_0\in [1,2)$ given be $\eqref{m0}$, the critical exponent  for the global existence (in time) of small data solutions  to \eqref{eq:CP0} is $p_c=1+\frac{2m\theta}{n}$. Moreover, we show that for $m\in (m_0, 2]$,  $p=p_c$
 belongs to the supercritical case, whereas  for   $m=1$, $p= 1+\frac{2\theta}{n}$ is expected to belong to the subcritical case. 
 
 The critical exponent for \eqref{eq:CP0} in the non-effective case $\alpha<2\theta$
will be discussed in a forthcoming paper.




\section*{Notation}

Through this paper, we use the following.
\begin{Not}
Let $f,g : \Omega \subset \R^n \to\R$ be two  functions. We use the notation $f \sim g$ if there exist two constants $C_1,C_2>0$ such that $C_1 g(y)\leq f(y) \leq C_2 g(y)$ for all $y\in \Omega$. If the inequality is one-sided, namely, if $f(y)\leq Cg(y)$ (resp.~$f(y)\geq Cg(y)$) for all~$y\in \Omega$, then we write~$f\lesssim g$ (resp.~$f\gtrsim g$).
\end{Not}
%
%
\begin{Not}
By~$[\,\cdot\,]: \R\to\N$, we denote the floor function:
\[ [x] \doteq \max \{ n\in\N:\,\, n\leq x\}. \]
By~$(x)_+$ we denote the positive part of~$x\in\R$, i.e. $(x)_+=\max\{x,0\}$.
\end{Not}
\begin{Not}
Let~$\chi_0, \chi_1$ be~$C^\infty(\R^{n})$ cut-off nonnegative functions satisfying
\[ \chi_0+\chi_1=1, \qquad \supp\chi_0\subset \{ \xii\leq 1\}, \qquad \text{and} \qquad \supp \chi_1 \subset \{ \xii\geq 1/2 \}. \]
In particular, it follows that~$\chi_0=1$ in~$\{\xii\leq 1/2\}$ and~$\chi_1=1$ in~$\{\xii\geq 1\}$. To localize a distribution~$g$ at low and high frequencies, we denote $g_{\chi_j}=\mathfrak{F}^{-1} (\chi_j\,\widehat{g})$, $j=0,1$.
\end{Not}
\begin{Not}
For any~$q\in[1,\infty]$, we denote by~$L^q(\R^n)$ the usual Lebesgue space over~$\R^n$. For any~$s\in[0,+\infty)$, we denote by~$H^{s,q}$ the Bessel potential space:
\[ H^{s,q}(\R^n) = \left\{ f\in L^q(\R^n) : \ (1-|D|)^sf\in L^q(\R^n) \right\}. \]
We recall that~$H^{s,q}(\R^n)=W^{s,q}(\R^n)$, the usual Sobolev space, for any~$q\in(1,\infty)$ and~$s\in\N$. As usual, we denote~$H^s(\R^n)\doteq H^{s,2}(\R^n)$ for any~$s\geq0$.
\end{Not}
\begin{Not}
	Let $f,g : \R^+ \times \Omega \to\R$ be two regular functions with  $\Omega \subset \R^n$. We use the notation $f \ast g$ to indicate the convolution with respect to the space variable of the functions $f$ and $g$, i.e.,
	$$ (f \ast g )(t,x) = \int_\Omega f(t,y) \cdot g(t,x-y)\,dy.$$
\end{Not}

\section{Main results}\label{sec:main}

Assuming data in a suitable space, one may conclude the local existence of solutions to~\eqref{eq:CP0} for $p>1$ (see \cite{CL}).
The  next result explain  that  for $1<p< p_c$ this solution can not exist globally in time even if the data are supposed to be very small.
\begin{Thm}
\label{thm:test0dec}
Let~$\delta, \theta\in\N$, $\alpha\in\N\setminus\{0\}$, and assume that~$u_0=0$, whereas~$u_1, (-\Delta)^\delta u_1\in L^1_{loc}(\R^n)$ verifies
\begin{equation}\label{eq:datatestdec}
u_1(x)+(-\Delta)^\delta u_1(x) \geq \varepsilon(1+|x|)^{-\frac{n}{m}}(\log \langle x\rangle)^{-1}, \qquad \text{for some~$\varepsilon\in(0,1)$ and $m\in (1, 2]$.}
\end{equation}
Then there exists no global (in time) weak solution to~\eqref{eq:CP0} for any
\[ p\in\left(1, 1+\frac{\min\{2\theta,\alpha\}}{n}m\right).\]
If
\[\int_{\R^n} ( I + (-\Delta)^\delta ) u_1(x)\, dx>0,\]
 then conclusion is still true for $m=1$.
\end{Thm}
\begin{Rem}
Hypothesis \eqref{eq:datatestdec} implies that $u_1+(-\Delta)^\delta u_1 \notin L^{m-\epsilon}(\R^n), $
for all $\epsilon>0$.
\end{Rem}
Let $n\geq 1$, $2\theta \in (0,\alpha]$ and let us define 
\begin{equation}\label{m0}
m_0\doteq \min \left\{ m \in [1, 2]; \,\,\, n(2-m) \leq  2m\theta\min\{m, \sqrt{2(2-m)}\}  \right\}.
 \end{equation}
\begin{Rem}
Conditions  \eqref{m0}  implies $\frac2{m}\leq 1+\frac{2m\theta}{n}$ and $1+\frac{2m\theta}{n}\leq \frac{n\,m}{2(n-2m\theta)_+}$ for all $m\in [m_0, 2]$ with $m_0 < 2$. But since  the last inequality should be  strict in  Theorem \ref{thm:L2p} and we are mainly interested in the case $m>1$, from now on we are going to   assume
  $m\in (m_0, 2]$.
\end{Rem}
\begin{Example}
If $n=1, 2$ and $\theta=1$, then $m_0=1$ and the admissible interval for  $m$  is $(1, 2]$.\\
If $n=3$ and  $\theta=\frac12$, then $m_0=\frac32$ and the admissible interval for  $m$  is $(\frac32, 2]$.
\end{Example}
 In the next  result we show that, under addition regularity $L^m (\R^n)$ for initial data,  global small data  solutions exist to \eqref{eq:CP0} for $2\theta\leq \alpha$ and
$  p_c\doteq 1+\frac{2m\theta}{n} < p \leq  \frac{n\,m}{2(n-2m\theta)_+}$.
 In this case, Theorem \ref{thm:test0dec} implies a nonexistence result  for $1<p< p_c$ and  we conclude that $p= p_c$ is the critical exponent for \eqref{eq:CP0}.
%
\begin{Thm}\label{thm:L2p}
Let~ $\delta\in(\theta,\alpha]$, $2\theta \in (0,\alpha]$ and   $m \in (m_0,2]$,  with $m_0$ given be $\eqref{m0}$. If
 \[  1+\frac{2m\theta}{n} <  p\leq \frac{q}{2} \leq \frac{n\,m}{2(n-2m\theta)_+},\]
then there exists a sufficiently small~$\varepsilon>0$ such that for any data
\begin{equation}\label{eq:datasp}
\begin{split}
(u_0,u_1) \in \mathcal A  \doteq  \left(H^{\alpha+\delta}(\R^n) \cap L^m (\R^n) \right)\times \left(H^{2\delta}(\R^n) \cap L^m (\R^n) \right),  \qquad \|(u_0, u_1)\|_{\mathcal A}\leq \varepsilon,\nonumber
\end{split}
\end{equation}
%
there exists a global (in time) energy solution $u\in\mathcal C([0,\infty),H^{\alpha}(\R^n)\cap L^{q}(\R^n)) \cap C^1([0,\infty), L^2(\R^n)\cap L^{q}(\R^n))$ to~\eqref{eq:CP0}.
Also, the solution to~\eqref{eq:CP0}  satisfies the estimates
\begin{eqnarray}
\label{eq:decayL2anew}
\||D|^{\alpha} u(t,\cdot)\|_{L^2} &\lesssim& (1+t)^{-\frac{1}{2(\alpha-\theta)}\left( n \left(\frac{1}{m}- \frac{1}{2}\right) +\alpha-2\theta\right)}
\|(u_0, u_1)\|_{\mathcal A},
\\
\label{eq:decayL2bnew}
\| \partial_t^j u(t,\cdot)\|_{L^{\kappa}} &\lesssim& (1+t)^{1-j-\frac{n}{2 \theta}\left(\frac{1}{m}- \frac{1}{\kappa}\right)}\|(u_0, u_1)\|_{\mathcal A},  \qquad  j=0,1,
\end{eqnarray}
for $ 2 \leq \kappa \leq q $.
\end{Thm}
In particular, if $n \leq 2m\theta$ and $q=+\infty$ we have the following result:
\begin{corollary}\label{thm:L2}
	Let~ $\delta\in(\theta,\alpha]$, $2\theta \in (0,\alpha]$, $m \in (1,2]$ and assume that $1\leq n \leq 2 m\theta$.
	Let~$p> 1+\frac{2 m\theta}{n}$, then there exists a sufficiently small~$\varepsilon>0$ such that for any data
	\begin{equation}\label{eq:dataspcor}
	\begin{split}
	(u_0,u_1) \in \mathcal A  \doteq  \left(H^{\alpha+\delta}(\R^n)\cap L^{m}(\R^n) \right)\times \left(H^{2\delta}(\R^n)\cap L^{m}(\R^n) \right),  \qquad \|(u_0, u_1)\|_{\mathcal A}\leq \varepsilon,\nonumber
	\end{split}
	\end{equation}
	%
	there exits a global (in time) energy solution $u\in\mathcal C([0,\infty),H^{\alpha}(\R^n)\cap L^{\infty}(\R^n)) \cap C^1([0,\infty), L^2(\R^n)\cap L^{\infty}(\R^n))$ to~\eqref{eq:CP0}.
	Also,  the solution to~\eqref{eq:CP0}  satisfies the estimates \eqref{eq:decayL2anew} and \eqref{eq:decayL2bnew}
	%
	for $\kappa\geq 2$.
\end{corollary}
\begin{Rem} Let $u_1, (-\Delta)^\delta u_1\in L^1_{loc}(\R^n)$  and  $2\theta\leq \alpha$. If
	\[\int_{\R^n} ( I + (-\Delta)^\delta ) u_1(x)\, dx>0,\] then Theorem \ref{thm:test0dec} implies that
	one can not have the existence of global  solutions to \eqref{eq:CP0} with $u_0=0$ and $p< 1+\frac{2\theta}{n}$.
If $u_1+(-\Delta)^\delta u_1 \notin L^{m-\epsilon}(\R^n), $ for all $\epsilon>0$ such that $m-\epsilon>1$,
then  Theorem \ref{thm:test0dec} implies that
	one can not have the existence of global solutions to \eqref{eq:CP0} with $u_0=0$ for $p < 1+\frac{2\theta(m-\epsilon)}{n}$ for all $\epsilon>0$.\\
This shows that, in general, the assumption $p>1+\frac{2m\theta}{n}$
	can not be removed in  Theorem \ref{thm:L2p}.
\end{Rem}
\begin{Rem}
	In Theorem \ref{thm:L2p} it appears a loss of regularity with respect to the initial data. This loss of regularity is related to the obtained estimates for solutions to the linear problem at high frequencies. In general one can not avoid this effect, for instance, if the initial data
	$u_1 \notin H^{2\delta}(\R^n)$, i.e., $u_1+(-\Delta)^\delta u_1 \notin L^{2}(\R^n), $ then, for $2\theta \leq \alpha$, the conclusion of Theorem \ref{thm:test0dec} is true for all $1<p < 1 +\frac{4\theta}{n}$ even if $u_1 \in L^{1}(\R^n)$.
	\end{Rem}
\begin{Rem}
The condition  $q \leq \frac{n\,m}{(n-2m\theta )_+}$ in Theorem \ref{thm:L2p} implies $n\left(\frac1m- \frac{1}{q}\right)\leq 2\theta$. Hence, with the assumed regularity for initial data in Theorem \ref{thm:L2p},  Theorem \ref{teo>} $(iii)$ implies that   solutions to the linear \eqref{eq:CPlinear} satisfies \eqref{eq:decayL2bnew} (see Remark \ref{usefulrem0}).
Similarly, the condition ~$n \leq 2 m\theta$ in Corollary \ref{thm:L2} is  in order that the $L^\infty$ norm of the
	partial derivative in time of solutions to the linear problem \eqref{eq:CPlinear}   has the decay given by $(1+t)^{-\frac{n}{2 m \theta}}$
	(see Theorem \ref{teo>}).
\end{Rem}

%
\begin{Example}  Let us assume that initial data has additional regularity  $L^{m}(\R^n)$. Then:\\
	The critical exponent for the beam equation with strong damping and rotational inertia effects, i.e., $n=1$, $\delta=1$, $\alpha=2$ and $\theta=1$
is $p_c=1+2m$.\\
	The critical exponent for plate equation with strong damping and rotational inertia effects, i.e., $n=2$, $\delta=1$, $\alpha=2$ and $\theta=1$ is $p_c=1+m$.\\
In both cases $(1, 2]$ is the admissible interval for  $m$.
\end{Example}
In the previous results,  one may feel the influence of additional regularity $L^m(\R^n)$, with $m\in (1,2]$, in the critical exponent.
In the case  $\theta =0$ we no longer have this effect, so in the next result we assume only data in the $L^2(\R^n)$ basis:
 \begin{Thm}\label{thm:theta0}
Let~ $\theta =0$, ~$0<\delta \leq\alpha$ and $n< 4\delta$.
Then, for all $p>1$ there exists a sufficiently small~$\varepsilon>0$ such that for any data
\begin{equation}\label{eq:dataspnew}
\begin{split}
(u_0,u_1) \in \mathcal A  \doteq  H^{\alpha+\delta}(\R^n)\times H^{2\delta}(\R^n),  \qquad \|(u_0, u_1)\|_{\mathcal A}\leq \varepsilon,\nonumber
\end{split}
\end{equation}
%
there exists a global (in time) energy solution $u\in\mathcal C([0,\infty),H^{\alpha}(\R^n)\cap L^{q}(\R^n)) \cap C^1([0,\infty), L^2(\R^n)\cap L^{q}(\R^n))$ to~\eqref{eq:CP0}, with $q\geq 2p$.
Also,  the solution to~\eqref{eq:CP0}  satisfies the estimates
\begin{eqnarray}
\label{eq:decayL2theta0}
\|\partial_t^j u(t,\cdot)\|_{L^2} &\lesssim& (1+t)^{-j}\|(u_0, u_1)\|_{\mathcal A}, \qquad  j=0, 1,
\\
\label{eq:decayL2theta01}
\||D|^{\alpha} u(t,\cdot)\|_{L^2} &\lesssim& (1+t)^{-\frac{1}{2}}
\|(u_0, u_1)\|_{\mathcal A},
\\
\label{eq:decayL2theta02}
\| u(t,\cdot)\|_{L^{q}} &\lesssim& (1+t)^{- \min\left\{\frac{n}{2\alpha}\left(\frac12- \frac{1}{q}\right), \frac{\alpha+\delta}{2\delta}- \frac{n}{2\delta}\left(\frac12- \frac{1}{q}\right)\right\}} \|(u_0, u_1)\|_{\mathcal A},
\\
\label{eq:decayL2theta03}
\| \partial_t u(t,\cdot)\|_{L^{q}} &\lesssim& (1+t)^{\frac{n}{2\delta}\left(\frac12- \frac{1}{q}\right)-1}\|(u_0, u_1)\|_{\mathcal A},
\end{eqnarray}
for all $q \geq 2p$.
\end{Thm}

\begin{Rem}
	As in previous theorems, estimates \eqref{eq:decayL2theta0} and \eqref{eq:decayL2theta01} coincide with  the  obtained estimates for solutions to the corresponding linear problem~\eqref{eq:CPlinear} at low frequencies. However, with the required regularity, estimate \eqref{eq:decayL2theta02} and \eqref{eq:decayL2theta03}  may coincide with the  obtained estimate for solutions to ~\eqref{eq:CPlinear}  at high frequencies.
\end{Rem}
In the next result we show that,  for initial data with additional regularity $L^{m}(\R^n)$,  $p=1+\frac{2 m\theta}{n}$ belongs to the supercritical case. For the classical damped wave equation, this phenomenon has been investigated in \cite{IIOW}.
\begin{Thm}\label{thm:sp}
Let~ $\delta\in(\theta,\alpha]$, $2\theta \in (0,\alpha]$ and   $m \in (m_0,2]$,  with $m_0$ given be $\eqref{m0}$.
		 Let~$p= 1+\frac{2 m\theta}{n} \leq \frac{q}{2}< \frac{n\,m}{2(n-2m\theta )_+}$.
		Then there exists a sufficiently small~$\varepsilon>0$ such that for any data
		\begin{equation}\label{eq:dataspcritical}
		\begin{split}
		(u_0,u_1) \in \mathcal A  \doteq  \left(H^{\alpha+\delta}(\R^n)\cap L^{m}(\R^n)\right) \times \left(H^{2\delta}(\R^n)\cap L^{m}(\R^n)\right),  \qquad \|(u_0, u_1)\|_{\mathcal A}\leq \varepsilon,
		\end{split}\nonumber
		\end{equation}
		%
		there exists a global (in time) energy solution $u\in\mathcal C([0,\infty),H^{\alpha}(\R^n)\cap L^{q}(\R^n)) \cap C^1([0,\infty), L^2(\R^n)\cap L^{q}(\R^n))$ to~\eqref{eq:CP0}.
		Also,  the solution to~\eqref{eq:CP0}  satisfies the estimates
		\begin{eqnarray}
		\label{eq:decayL2a}
		\||D|^{\alpha} u(t,\cdot)\|_{L^2} &\lesssim& \left\{ \begin{array}{ll} (1+t)^{-\frac{1}{2(\alpha-\theta)}\left( n \left(\frac{1}{m}- \frac{1}{2}\right) +\alpha-2\theta\right)}
		\|(u_0, u_1)\|_{\mathcal A}, \hspace{1,87cm} \text{if } \alpha = 2 \theta  \nonumber\\ (1+t)^{-\frac{1}{2(\alpha-\theta)}\left( n \left(\frac{1}{m}- \frac{1}{2}\right) +\alpha-2\theta\right)}
		\log(e+t)\|(u_0, u_1)\|_{\mathcal A}, \quad \text{if } \alpha > 2 \theta, \end{array} \right.
		\\
		\label{eq:decayL2b}
		\| \partial_t^j u(t,\cdot)\|_{L^{\kappa}} &\lesssim& (1+t)^{1-j-\frac{n}{2 \theta}\left(\frac{1}{m}- \frac{1}{\kappa}\right)}\|(u_0, u_1)\|_{\mathcal A},  \qquad  j=0,1,\nonumber
		\end{eqnarray}
		for all $2\leq \kappa \leq q$.
\end{Thm}

\section{ Non-existence via test function method}
For the  proof  of the next result one may follow as in \cite{DAE17NA} (see also \cite{DAL03}), but in order to explain the influence of the rotational inertia term we sketch the proof.
\begin{proof}(Theorem~\ref{thm:test0dec})
We fix a nonnegative, non-increasing, test function~$\varphi\in\mathcal{C}_c^\infty([0,\infty))$ with~$\varphi=1$ in~$[0,1/2]$ and~$\supp\varphi\subset[0,1]$, and a  nonnegative, radial, test function~$\psi\in\mathcal{C}_c^\infty(\R^n)$, such that~$\psi=1$ in the ball~$B_{1/2}$, and~$\supp\psi\subset B_1$. We also assume~$\psi(x)\leq\psi(y)$ when~$|x|\geq|y|$. Here~$B_r$ denotes the ball of radius~$r$, centered at the origin. We may assume (see, for instance, \cite{DAL03, MP}) that
\begin{equation}\label{eq:testbnd}
\varphi^{-\frac{p'}{p}}\,|\varphi'|^{p'}, \qquad\psi^{-\frac{p'}{p}} \bigl(|\Delta^\delta \psi|^{p'}+|\Delta^\theta \psi|^{p'}
+|\Delta^\alpha \psi|^{p'}\bigr), \qquad \text{are bounded,}
\end{equation}
where~$p'=p/(p-1)$. We remark that the assumption that~$\delta, \theta$ and~$\alpha$ are integers plays a fundamental role here. Then, for~$R\geq1$, we define:
\begin{equation}\label{eq:testR}
\varphi_R(t) = \varphi(R^{-\kappa}t), \quad \psi_R(x)=\psi(R^{-1}x),\nonumber
\end{equation}
for some~$\kappa>0$ which we will fix later.

Let~$\Phi_R\in\mathcal{C}_c^\infty([0,\infty))$ be the test function defined by
\[ \Phi_R(t)=\int_t^\infty\varphi_R(s)\,ds. \]
(Indeed, we notice that~$\supp\Phi_R\subset[0,R^\kappa]$, since~$\supp\varphi_R\subset[0,R^\kappa]$). In particular, $\Phi_R'=-\varphi_R$.

Let us assume that~$u\in L^1_{loc}([0, T]\times \R^n)$, with $u_t\in L^p_{loc}([0, T]\times \R^n)$ is a (local or global) weak solution to~\eqref{eq:CP0}. Let~$R>0$, and also assume that~$R\leq T^\kappa$, if~$u$ is a local solution in~$[0,T]\times\R^n$. Integrating by parts, and recalling that~$u_0=0$ and~$\varphi_R(0)=1$, we obtain
\begin{eqnarray}\label{eq:test0}
I_R&=&\int_0^\infty \int_{\R^n} u_t \bigl(-\varphi_R'\psi_R+\varphi_R(-\Delta)^\theta \psi_R+\Phi_R(-\Delta)^\alpha\psi_R-
\varphi_R'(-\Delta)^\delta \psi_R\bigr)\,dxdt \nonumber\\
 &&- \int_{\R^n} \psi_R(x)(I+ (-\Delta)^\delta) u_1(x)\,\,dx ,\nonumber
\end{eqnarray}
where:
\[ I_R=\int_0^\infty \int_{\R^n} |u_t|^p\varphi_R\psi_R\,dxdt. \]
We may now apply Young inequality to estimate:
\begin{multline*}
\int_0^\infty \int_{\R^n} |u_t| \bigl(|\varphi_R'|\,\psi_R+|\varphi_R|\,|(-\Delta)^\theta \psi_R|+\Phi_R\,|(-\Delta)^\alpha\psi_R|+|\varphi_R'(-\Delta)^\delta \psi_R|\bigr)\,dxdt \\
   \leq \frac1p\,I_R+\frac1{p'}\,\int_0^\infty \int_{\R^n} (\varphi_R\psi_R)^{-\frac{p'}{p}} \bigl(|\varphi_R'\psi_R|+|\varphi_R(-\Delta)^\theta \psi_R|+|\Phi_R(-\Delta)^\alpha \psi_R|+ |\varphi_R'(-\Delta)^\delta \psi_R|\bigr)^{p'}dxdt.
\end{multline*}
 Due to
\begin{gather*}
\varphi_R'(t) = R^{-\kappa}\varphi'(R^{-\kappa}t), \qquad
(-\Delta)^k\psi_R(x)=R^{-2k}\bigl((-\Delta)^k\psi\bigr)(R^{-1}x),
\end{gather*}
recalling~\eqref{eq:testbnd}, we may estimate
\begin{eqnarray*}
\int_0^\infty \int_{\R^n} (\varphi_R\psi_R)^{-\frac{p'}{p}} |\varphi_R'\psi_R|^{p'}\,dxdt
    & \leq& C\,R^{-\kappa p'+n+\kappa}\,,\\
    \int_0^\infty \int_{\R^n} (\varphi_R\psi_R)^{-\frac{p'}{p}} |\varphi_R(-\Delta)^\theta \psi_R|^{p'}\,dxdt
    & \leq& C\,R^{-2\theta p'+n+\kappa}\,,\\
\int_0^\infty \int_{\R^n} (\varphi_R\psi_R)^{-\frac{p'}{p}} |\varphi_R'(-\Delta)^\delta \psi_R|^{p'}\,dxdt
    & \leq& C\,R^{-(\kappa+2\delta) p'+n+\kappa}\,.
\end{eqnarray*}
Due to~$\Phi_R(t) \leq \Phi_R(0)\leq R^\kappa$, and being $\Phi_R^{p'}\,\varphi_R^{-\frac{p'}{p}}$ bounded  one gets:
\[ \int_0^\infty \int_{\R^n} (\varphi_R\psi_R)^{-\frac{p'}{p}} |\Phi_R(-\Delta)^\alpha\psi_R|^{p'}\,dxdt \leq C\,R^{-2\alpha p'+\kappa p'+ n+\kappa}\,.\]
We may now fix~$\kappa=\min\{2\theta,\alpha\}$, so that, summarizing, we proved that
\[ \frac1{p'}\,I_R\leq C\,R^{-\kappa p' +n+\kappa} -\int_{\R^n}\psi_R(x)(I+ (-\Delta)^\delta) u_1(x)\,\,dx.\]
 Recalling assumption~\eqref{eq:datatestdec}, there exists $c>0$ such that
 \[ \int_{\R^n}\psi_R(x)(I+ (-\Delta)^\delta) u_1(x)\,\,dx \geq \varepsilon \int_{\R^n} (1+|x|)^{-\frac{n}{m}}(\log \langle x\rangle)^{-1}\,\psi_R(x)\,dx \geq c\varepsilon\,R^{n-\frac{n}{m}}(\log \langle R\rangle)^{-1}. \]
As a consequence:
\[ \frac1{p'}\,I_R \leq C\,R^{-\kappa p' +n+\kappa} - c\varepsilon\,R^{n-\frac{n}{m}}(\log \langle R\rangle)^{-1} = R^n\, \bigl(C\,R^{-(p'-1)\kappa}-c\varepsilon\,R^{-\frac{n}{m}}(\log \langle R\rangle)^{-1}\bigr).\]
Assume, by contradiction, that the solution~$u$ is global. In the subcritical case~$p< 1+\frac{\min\{2\theta,\alpha\}}{n}m$, it follows that~$(p'-1)\kappa>\frac{n}{m}$ and $I_R<0$, for any sufficiently large~$R$, and this contradicts the fact that~$I_R\geq0$.
 Therefore, $u$ cannot be a global (in time) solution  and this concludes the proof.
\end{proof}

\section{The linear estimates}\label{sec:linear}

We consider the inhomogeneous linear problem
\begin{equation}\label{eq:CPlinearinho}
\begin{cases}
u_{tt}+(-\Delta)^\delta u_{tt}+(-\Delta)^\alpha u+(-\Delta)^\theta u_t=f(t,x), \quad t\geq 0,\,\, x\in\R^n,\\
(u,u_t)(0,x)=(u_0,u_1)(x).
\end{cases}
\end{equation}
We introduce the Fourier multipliers
\begin{equation}\label{K0K1}
\hat{K}_0(t,\xi)=\dfrac{\lambda_{+}e^{t\lambda_{-}}
-\lambda_{-}e^{t\lambda_{+}}}{\lambda_{+}-\lambda_{-}}
\hspace{0,8cm} \text{and } \hspace{0,8cm}
\hat{K}_1(t,\xi)=\dfrac{e^{t\lambda_{+}}-e^{t\lambda_{-}}}{\lambda_{+}-\lambda_{-}},
\end{equation}
with
\[\lambda_{\pm}=\frac{|\xi|^{2\theta}}{2(1+|\xi|^{2\delta})}\left(-1
\pm \sqrt{1-4|\xi|^{2(\alpha-2\theta)}(1+|\xi|^{2\delta})}\,\right).\]
The solution to \eqref{eq:CPlinearinho} may be written as
\[
u(t, x)=K_0(t)\ast u_0+K_1(t)\ast u_1+ \int_0^t E_1(t-s, x) \ast \, f(s,x)\, ds,
\]
where $K_0(t,x)=\mathcal{F}^{-1}[\hat{K}_0(t,\cdot)](x)$, $K_1(t,x)=\mathcal{F}^{-1}[\hat{K}_1(t,\cdot)](x)$ and
\[E_1(t, x)=\left(I+(-\Delta)^\delta\right)^{-1}K_1(t,x).\]
Some  estimates in the following result was  already discussed in   \cite{CP} for $\eta=1$, but in order to deal with the semilinear problem  we had to derive
estimates for a large range of parameters:
\begin{Thm}\label{teo>} Let $\alpha>0$, $2\theta\leq \alpha$, $\delta \in [0,\alpha]$, $\eta\in [1,2]$,  $q \in [2,+\infty]$
	and $j=0,1$. Then the kernels defined by \eqref{K0K1} satisfy: \\
	\noindent (i)
	\[
	||\partial_x^{\gamma_2}\partial_t^j K_0(t, \cdot)\ast \psi||_{L^q}\lesssim
	(1+t)^{-\frac{1}{2(\alpha-\theta)}\left(n\left(\frac1{\eta}-\frac1q\right)+|\gamma_2|\right)-j}||\psi||_{L^{\eta}}+
	g(t)||\psi||_{H^{s_j}}.
	\]
	\vspace{0,3cm}\noindent (ii) If  $n\left(\frac1{\eta}-\frac1{q}\right)+|\gamma_2|-2\theta\geq0$ then
	\[
	||\partial_x^{\gamma_2}\partial_t^j K_1(t, \cdot)\ast \psi||_{L^q} \lesssim
	(1+t)^{-\frac{1}{2(\alpha-\theta)}\left(n\left(\frac1{\eta}-\frac1q\right)
		+|\gamma_2|-2\theta\right)-j}||\psi||_{L^{\eta}}
	+	g(t)||\psi||_{H^{r_j}}\]
	and
	\[
	||\partial_x^{\gamma_2}\partial_t^j E_1(t, \cdot)\ast \psi||_{L^q}\lesssim
	(1+t)^{-\frac{1}{2(\alpha-\theta)}\left(n\left(\frac1{\eta}-\frac1q\right)
		+|\gamma_2|-2\theta\right)-j}||\psi||_{L^{\eta}}
	+	g(t)||\psi||_{H^{(r_j-2\delta)_+}}.
	\]
A special exception is given in the case   $j=0$, $\eta=1$, $q\geq 2$ and  $n\left(1-\frac1{q}\right)+|\gamma_2|-2\theta = 0$,
 namely,
\[
	||\partial_x^{\gamma_2} K_1(t, \cdot)\ast \psi||_{L^q}\lesssim
	\ln(e+t)||\psi||_{L^{1}}
	+	g(t)||\psi||_{H^{r_0}},
\]
and
\[
	||\partial_x^{\gamma_2} E_1(t, \cdot)\ast \psi||_{L^q}\lesssim
	\ln(e+t)||\psi||_{L^{1}}
	+g(t)||\psi||_{H^{(r_0-2\delta)_+}}.
	\]\\
	\vspace{0,3cm}\noindent (iii) If  $n\left(\frac1{\eta}-\frac1{q}\right)+|\gamma_2|-2\theta < 0$,
 	then
	\[
	||\partial_x^{\gamma_2}\partial_t^j K_1(t, \cdot)\ast \psi||_{L^q}\lesssim
	(1+t)^{1-j-\frac{1}{2\theta}\left(n\left(\frac1{\eta}-\frac1q\right)+|\gamma_2|\right)}||\psi||_{L^{\eta}}
	+	g(t)||\psi||_{H^{r_j}}
	\]
	and
	\[
	||\partial_x^{\gamma_2}\partial_t^j E_1(t, \cdot)\ast \psi||_{L^q}\lesssim
	(1+t)^{1-j-\frac{1}{2\theta}\left(n\left(\frac1{\eta}-\frac1q\right)+|\gamma_2|\right)}||\psi||_{L^{\eta}}
	+g(t)||\psi||_{H^{(r_j-2\delta)_+}}.
	\] \vspace{-0,4cm}\\
	Here $g(t) = \left\{
	\begin{array}{ll} e^{-ct}\,\,(c \in \R^+), \hspace{2,78cm}
	\text{if} \,\,\,\delta \leq \theta  \\
	(1+t)^{\frac{n}{2(\delta-\theta)}\left( \frac{1}{2} - \frac{1}{q} \right)}(1+t)^{-\frac{1}{2\beta}}, \hspace{0,5cm} \text{if} \,\,\,\theta <
	\delta \hspace{0,25cm} \text{with}  \hspace{0,25cm} 0 < \beta < \frac{\delta-\theta}{n} \frac{2q}{(q-2)_+},\end{array} \right.$
	\vspace{0,2cm}\\
	\noindent$s_0 = \left\{
	\begin{array}{ll} |\gamma_2 |\hspace{1,52cm}
	\text{if} \,\,\,\delta \leq \theta \\
	|\gamma_2| + \frac{\delta-\theta}{\beta} \hspace{0,5cm} \text{if}
	\,\,\,\theta < \delta \end{array} \right.$, $\,\,\,r_0=s_0+\delta-\alpha$ \hspace{0,25cm}and\hspace{0,25cm}
	$r_1 = \left\{
	\begin{array}{ll} |\gamma_2 |\hspace{1,52cm}
	\text{if} \,\,\,\delta \leq \theta \\
	|\gamma_2| + \frac{\delta-\theta}{\beta} \hspace{0,5cm} \text{if}
	\,\,\,\theta < \delta \end{array} \right.$,   $\,\,\,s_1=r_1 +\alpha
	-\delta$.
\end{Thm}
\begin{proof}    For small frequencies $|\xi|\leq 1$ we have that
	\[\lambda_{+} \approx -|\xi|^{2(\alpha-\theta)}, \qquad \lambda_{-} \approx -|\xi|^{2\theta}, \qquad \lambda_{+}-\lambda_{-}\approx  |\xi|^{2\theta},\]
	and if $t|\xi|^{2\theta}\leq 1$ we have
	\[|\hat{K}_1(t,\xi)|\lesssim |\xi|^{-2\theta}e^{t\lambda_{-}}
	\left(e^{t(\lambda_{+}-\lambda_{-})}-1\right)\lesssim t e^{t\lambda_{-}},\]
	whereas for $t|\xi|^{2\theta}\geq 1$
	\[|\hat{K}_1(t,\xi)| \lesssim |\xi|^{-2\theta} e^{t\lambda_{+}}.\]
	For any~$\eta\in[1,2]$, we define~$\eta'=\eta/(\eta-1)$, its H\"older conjugate, and~$r\in[2,\infty]$ by
	\begin{equation}\label{eq:r}
	\frac1r = \frac1{q'}-\frac1{\eta'} = \frac1\eta-\frac1q,
	\end{equation}
	where $q'=q/(q-1)$.
	Now, Hausdorff-Young inequality comes into play
		\[
\| (-\triangle)^{\frac{k}{2}}\mathfrak F^{-1}(\chi_0\hat{K}_1(t,\cdot)\hat \psi)\|_{L^{q}}\lesssim\|\,|\xi|^{k} \hat{K}_1(t,\cdot)\hat \psi\|_{L^{q'}(|\xi|\leq 1)}
	\]
and we estimate $L^{q'}$ in the regions $t|\xi|^{2\theta}\leq 1$ and  $t^{-\frac1{2\theta}}\leq |\xi|\leq 1$:
\begin{align*}
	\|\,|\xi|^{k} \hat{K}_1(t,\cdot)\hat \psi\|_{L^{q'}(t|\xi|^{2\theta}\leq 1)}
	& \lesssim t\,\|\,|\xi|^{k}e^{t\lambda_{-}}\|_{L^r(t|\xi|^{2\theta}\leq 1)}\,\|\hat \psi\|_{L^{\eta'}}\\
	& \lesssim (1+t)^{1-\frac1{2\theta}\left(\frac{n}{r}+k\right)} \,\|\psi\|_{L^\eta},
	\end{align*}
	whereas
	\begin{align*}
		\|\,|\xi|^{k} \hat{K}_1(t,\cdot)\hat \psi\|_{L^{q'}(t^{-\frac1{2\theta}}\leq |\xi|\leq 1)}
		& \lesssim  \|\,|\xi|^{k-2\theta}e^{t\lambda_{+}}\|_{L^r(t^{-\frac1{2\theta}}\leq |\xi|\leq 1)}\,\|\hat \psi\|_{L^{\eta'}}\\
		& \lesssim  (1+t)^{-\frac1{2(\alpha-\theta)}\left(\frac{n}{r}+k-2\theta\right)} \,\|\psi\|_{L^\eta},
		\end{align*}
for $n\left(\frac1{\eta}-\frac1{q}\right)+k-2\theta>0$,
\begin{align*}
		\|\,|\xi|^{k} \hat{K}_1(t,\cdot)\hat \psi\|_{L^{q'}(t^{-\frac1{2\theta}}\leq |\xi|\leq 1)}
		& \lesssim  \|\,|\xi|^{k-2\theta}e^{t\lambda_{+}}\|_{L^r(t^{-\frac1{2\theta}}\leq |\xi|\leq 1)}\,\|\hat \psi\|_{L^{\eta'}}\\
		& \lesssim  (1+t)^{1-\frac1{2\theta}\left(\frac{n}{r}+k\right)} \,\|\psi\|_{L^\eta},
		\end{align*}
for $n\left(\frac1{\eta}-\frac1{q}\right)+k-2\theta<0$, and
\begin{align*}
		\|\,|\xi|^{k} \hat{K}_1(t,\cdot)\hat \psi\|_{L^{q'}(t^{-\frac1{2\theta}}\leq |\xi|\leq 1)}
		& \lesssim  \|\,|\xi|^{k-2\theta}e^{t\lambda_{+}}\|_{L^r(t^{-\frac1{2\theta}}\leq |\xi|\leq 1)}\,\|\hat \psi\|_{L^{\eta'}}\\
		& \lesssim   \log(e+t)\,\|\psi\|_{L^\eta},
		\end{align*}
for $n\left(\frac1{\eta}-\frac1{q}\right)+k-2\theta=0$.
However, in some cases the last inequality may be improved.
 Indeed, the case $k-2\theta=0$    is immediately,  i.e.,
 \[ \| (-\triangle)^{\frac{k}{2}}\mathfrak F^{-1}(\chi_0\hat{K}_1(t,\cdot)\hat \psi)\|_{L^{2}}  \lesssim   \,\|\psi\|_{L^2}.
 \]
 So, let us suppose that $k-2\theta<0$.  If $2\leq q <\infty$ and $1< \eta < \frac{n}{2\theta-k}$,
 by using the  Riesz potential mapping properties $I_{2\theta-k}f=\mathfrak F^{-1}(|\xi|^{-(2\theta-k)}\hat f)$  we get \cite{Stein}
\begin{align*}
\| (-\triangle)^{\frac{k}{2}}\mathfrak F^{-1}(\chi_0\hat{K}_1(t,\cdot)\hat \psi)\|_{L^{q}}&=\|I_{2\theta-k} (-\triangle)^{\frac{k}{2}}\mathfrak F^{-1}(\chi_0|\xi|^{2\theta-k}\hat{K}_1(t,\cdot)\hat \psi)\|_{L^{q}}\\
&\lesssim\| \mathfrak F^{-1}(\chi_0|\xi|^{2\theta}\hat{K}_1(t,\cdot)\hat \psi)\|_{L^{\eta}}, \qquad \frac1{\eta}-\frac1{q}= \frac{2\theta-k}{n}.
\end{align*}
For  $\beta\in \N^n$  we  may estimate
\[|\partial_\xi^\beta \chi_0|\xi|^{2\theta}\hat{K}_1(t, \xi)| \lesssim |\xi|^{-|\beta|},\]
hence,  Mikhlin-H\"ormander multiplier theorem implies
 \[
\| \mathfrak F^{-1}(\chi_0|\xi|^{2\theta}\hat{K}_1(t,\cdot)\hat \psi)\|_{L^{\eta}}
  \lesssim  \|\psi\|_{L^{\eta}},
\]
and
\[\| (-\triangle)^{\frac{k}{2}}\mathfrak F^{-1}(\chi_0\hat{K}_1(t,\cdot)\hat \psi)\|_{L^{q}}\lesssim  \|\psi\|_{L^{\eta}}.\]
If $q=\infty$ and $\eta>1$,  by taking $s>0$ such that $2\theta-k-1<s< \frac{n}{\eta}$  and applying Lemma 3.1 of \cite{DAE14JDE}
we have
 \begin{align*}
\| (-\triangle)^{\frac{k}{2}}\mathfrak F^{-1}(\chi_0\hat{K}_1(t,\cdot)\hat \psi)\|_{L^{\infty}}&=\|\mathfrak F^{-1}(\chi_0|\xi|^{k+s}\hat{K}_1(t,\cdot)|\xi|^{-s}\hat \psi)\|_{L^{\infty}}\\
		& \lesssim (1+t)^{-\frac1{2(\alpha-\theta)}\left(\frac{n}{\tilde q}+k+s-2\theta\right)} \|\, I_{s}\psi\|_{L^{\tilde q}}\\
  &\lesssim  \|\psi\|_{L^{\eta}}, \qquad \frac1{\eta}-\frac1{\tilde q}= \frac{s}{n},
\end{align*}
thanks to $\frac{n}{\eta}+k-2\theta=0$.
		Furthermore, we have
		\begin{align*}
		\|\,|\xi|^{k} \hat{K}_{0}(t,\cdot)\hat \psi\|_{L^{q'}(|\xi|\leq 1)}
		& \lesssim \|\,|\xi|^{k}e^{t\lambda_{+}}\|_{L^r(|\xi|\leq 1)}\,\|\hat \psi\|_{L^{\eta'}}\\
		& \lesssim (1+t)^{-\frac1{2(\alpha-\theta)}\left(\frac{n}{r}+k\right)} \,\|\psi\|_{L^\eta}.
		\end{align*}
	For  time derivatives  of the kernels the desired  estimates follows thanks to
	\[|\partial_t\hat{K}_0(t,\xi)|=|\lambda_{+}\lambda_{-}|\left|\dfrac{e^{t\lambda_{-}}
		-e^{t\lambda_{+}}}{\lambda_{+}-\lambda_{-}}\right|\lesssim |\xi|^{2(\alpha-\theta)}e^{t\lambda_{+}}\]
	\[|\partial_t\hat{K}_1(t,\xi)|=\left|\dfrac{\lambda_{-}e^{t\lambda_{-}}
		-\lambda_{+}e^{t\lambda_{+}}}{\lambda_{+}-\lambda_{-}}\right|\lesssim e^{t\lambda_{-}} + |\xi|^{2(\alpha-2\theta)}e^{t\lambda_{+}}.\]
	In the low frequency region the estimates obtained for $\hat{E}_1$ and $\hat{K}_1$ are the same, since $\hat{E}_1 \approx \hat{K}_1$ for all $t \geq 0$.
	
	At high frequencies the roots of the full symbol are complex-valued, and
	\begin{equation}\label{eq:autovalores-alta} Re \lambda_{\pm} \approx -|\xi|^{2(\theta-\delta)}, \qquad |\lambda_{\pm}| \approx |\xi|^{\alpha-\delta}, \qquad |\lambda_{+}-\lambda_{-}| \approx  |\xi|^{\alpha-\delta}, \qquad |\xi| \rightarrow \infty.\end{equation}
Using the equivalences \eqref{eq:autovalores-alta}, we have
		\[|\partial_t^j\hat{K}_\ell(t,\xi)|
		\lesssim  |\xi|^{(j-\ell)(\alpha-\delta)}e^{t Re \lambda_{\pm}}.\]
	Hence, if $\theta\geq \delta$ and $u_0\in H^{\alpha}(\R^n), u_1\in H^{\delta}(\R^n)$, we have an exponential decay for $\||\xi|^b \partial_t^j \hat u(t,\cdot)\|_{L^q(|\xi|\geq 1)}$ for~$0\leq b+ (\alpha-\delta)j\leq \alpha$,  $j=0,1$.

	On the other hand, if $\delta>\theta$, for $\beta>0$ we may estimate
	\begin{align*} e^{-t |\xi|^{2(\theta-\delta)}}&= t^{-\frac1{2\beta}} ( t |\xi|^{2(\theta-\delta)})^\frac1{2\beta}e^{-t |\xi|^{2(\theta-\delta)}}|\xi|^{\frac{\delta-\theta}{\beta}}\\
	& \lesssim  t^{-\frac1{2\beta}} |\xi|^{\frac{\delta-\theta}{\beta}}, \qquad t>0
	\end{align*}
	so, under additional regularity $\frac{\delta-\theta}{\beta}$ on initial data, we have a polynomial decay
	$(1+t)^{-\frac1{2\beta}}$ for the $L^2$ norm
		\begin{align*}
		||\,|\xi|^{k}\partial_t^j \hat{K}_\ell(t, \cdot) \hat{\psi}||_{L^2(|\xi|\geq 1)}&\lesssim || \,|\xi|^{k+(j-\ell)(\alpha-\delta)}e^{t Re \lambda_{\pm}} \hat{\psi}||_{L^2(|\xi|\geq 1)}\\
		&\lesssim ||\, |\xi|^{k+(j-\ell)(\alpha-\delta)+\frac{\delta-\theta}{\beta}}t^{-\frac{1}{2\beta}}\hat{\psi}||_{L^2(|\xi|\geq 1)} \\
		&\lesssim t^{-\frac{1}{2\beta}} ||\psi||_{H^{s}},
		\end{align*}
with $s=s_j$ if $\ell = 0$ and $s=r_j$ if $\ell = 1$.
			For $q\geq 2$ and
$q'=q/(q-1)$, we may use again
	 Hausdorff-Young inequality
		\[
\| (-\triangle)^{\frac{k}{2}}\partial_t^j\mathfrak F^{-1}(\chi_1\hat{K}_\ell(t,\cdot)\hat \psi)\|_{L^{q}}=\|\,|\xi|^{k}\partial_t^j \hat{K}_\ell(t,\cdot)\hat \psi\|_{L^{q'}(|\xi|\geq 1)},
	\]
$\ell,j = 0,1$.
Since $q' \in [1,2)$ and $n < \left[ -k + s - (j-\ell)(\alpha - \delta) \right]\,\frac{2q'}{2-q'}$
		\begin{align*}||\, |\xi|^k \partial_t^j \hat{K}_{\ell}(t, \cdot) \hat{\psi}||_{L^{q'}(|\xi|\geq 1)}&\lesssim || \,|\xi|^{k+(j-\ell)(\alpha-\delta)}e^{t Re \lambda_{\pm}} \hat{\psi}||_{L^{q'}(|\xi|\geq 1)}\\
		&\lesssim ||\, |\xi|^{k-s+(j-\ell)(\alpha-\delta)}e^{t Re \lambda_{\pm}} ||_{L^\frac{2q'}{2-q'}(|\xi|\geq 1)} ||\,|\xi|^{s}\hat{\psi}||_{L^2(|\xi|\geq 1)}\\
		&\lesssim (1+t)^{\frac{1}{2(\delta-\theta)} \left[ n \left( \frac{2-q'}{2q'} \right) + k -s + (j-\ell)(\alpha - \delta) \right]} ||\psi||_{H^{s}}.
		\end{align*}
		Now calling $s=\frac{\delta-\theta}{\beta}+k+(j-\ell)(\alpha-\delta)$ (with $\beta>0$), if $n<\frac{2q'(\delta - \theta)}{(2-q')\beta}$ and $q' \in [1,2)$, it follows that
		\[|| \,|\xi|^k \partial_t^j \hat{K}_{\ell}(t, \cdot) \hat{\psi}||_{L^{q'}(|\xi|\geq 1)} \lesssim  (1+t)^{\frac{n(2-q')}{4q'(\delta-\theta)}-\frac1{2\beta}} ||\psi||_{H^{s}}\lesssim  g(t)||\psi||_{H^{s}}, \qquad \ell = 0,1,\]
with $s=s_j$ if $\ell = 0$ and $s=r_j$ if $\ell = 1$.
			\end{proof}
\begin{Rem}\label{usefulrem0}
We assume the hypotheses in  Theorem \ref{thm:L2p}, then  solutions to \eqref{eq:CPlinearinho}, with $f\equiv 0$, satisfy \eqref{eq:decayL2anew}-\eqref{eq:decayL2bnew}. Indeed, the assumed maximum regularity $H^{r}$ for the second data, with $r=\max\{r_0, r_1\}$,   is in order that
\[|||D|^{\alpha} K_1(t, \cdot)\ast \psi||_{L^2} \lesssim
(1+t)^{-\frac{1}{2(\alpha-\theta)}\left(n\left(\frac1{m}-\frac12\right)
	+\alpha-2\theta\right)}||\psi||_{ H^{r_0}\cap L^m}\]
for $n \left( \frac{1}{m} - \frac{1}{q} \right) \leq 2 \theta$, i.e.,
\[ r_0=\delta+\frac{\delta-\theta}{\beta} =\delta+ \frac{\delta-\theta}{\alpha-\theta} \left(n\left(\frac1{m}-\frac12\right)
+\alpha-2\theta\right) \leq \delta+ \frac{\delta-\theta}{\alpha-\theta} \left(n\left(\frac1{m}-\frac1q\right)
+\alpha-2\theta\right) \leq \delta+ \frac{\delta-\theta}{\alpha-\theta}\,\alpha \leq 2\delta,\]
 and
\[|| \partial_t K_1(t, \cdot)\ast \psi||_{L^{q}}
\lesssim (1+t)^{- \frac{n}{2\theta}\left(\frac{1}{m} - \frac{1}{q} \right)}||\psi||_{ H^{r_1}\cap L^m},\]
i.e.,
\[ r_1=\frac{\delta-\theta}{\beta}=(\delta-\theta)\left[\dfrac{n}{\delta - \theta} \left( \dfrac{1}{2} - \dfrac{1}{q} \right) + \dfrac{n}{\theta} \left( \dfrac{1}{m} - \dfrac{1}{q} \right)\right]
\leq n \left( \dfrac{1}{m} - \dfrac{1}{q} \right) + \dfrac{n(\delta-\theta)}{\theta} \left( \dfrac{1}{m} - \dfrac{1}{q} \right) \leq 2\delta.\]
This implies the required  regularity $H^{s}$  for the first data, with $s=\max\{s_0, s_1\}$, since $s_j=r_j + \alpha -\delta$, $j=0,1$.
\end{Rem}
\begin{Rem}\label{usefulremnew}
		We assume the hypotheses in  Theorem \ref{thm:theta0}, then solutions to \eqref{eq:CPlinearinho}, with $f\equiv 0$, satisfy estimates \eqref{eq:decayL2theta02} and \eqref{eq:decayL2theta03}. Indeed, if $(u_0,u_1) \in H^{s_1}(\R^n) \times H^{r_1}(\R^n)$, with $s_1=\alpha+\delta$ and $r_1=2\delta$, we have
		\begin{align}\|u_t(t,\cdot)\|_{L^q} & \lesssim (1+t)^{-\frac{n}{2\alpha} \left( \frac{1}{2}-\frac{1}{q} \right)-1} \left(\|u_0\|_{L^2} + \|u_1\|_{L^2} \right) + (1+t)^{\frac{n}{2\delta} \left( \frac{1}{2}-\frac{1}{q} \right)-\frac{1}{2\beta}} \left(\|u_0\|_{H^{s_1}} + \|u_1\|_{H^{r_1}} \right) \nonumber \\
		& \lesssim (1+t)^{\frac{n}{2\delta} \left( \frac{1}{2}-\frac{1}{q} \right)-1} \left(\|u_0\|_{H^{s_1} \cap L^2} + \|u_1\|_{H^{r_1} \cap L^2} \right) \nonumber
		\end{align}
	because $r_1 = \frac{\delta}{\beta}$, i.e., $\beta = \frac12$ implies
	$\frac{n}{2\delta} \left( \frac{1}{2}-\frac{1}{q} \right)-\frac{1}{2\beta} \geq -\frac{n}{2\alpha} \left( \frac{1}{2}-\frac{1}{q} \right)-1.$
	
	On the other hand, if $(u_0,u_1) \in H^{s_0}(\R^n) \times H^{r_0}(\R^n)$, with $s_0=\alpha+\delta$ and $r_0=2\delta$, we have
	\begin{align}\|u(t,\cdot)\|_{L^q} & \lesssim (1+t)^{-\frac{n}{2\alpha} \left( \frac{1}{2}-\frac{1}{q} \right)} \left(\|u_0\|_{L^2} + \|u_1\|_{L^2} \right) + (1+t)^{\frac{n}{2\delta} \left( \frac{1}{2}-\frac{1}{q} \right)-\frac{1}{2\beta}} \left(\|u_0\|_{H^{s_0}} + \|u_1\|_{H^{r_0}} \right), \nonumber
	\end{align}
	where $\beta = \frac{\delta}{\alpha+\delta}$.
	We note that, depending on the the parameters $\alpha, \delta$ and in the space dimension $n$,  the last estimate may be determined at  low frequency or at high frequency.
\end{Rem}

\section{ Proof of Theorems \ref{thm:L2p} to  \ref{thm:sp} }

By Duhamel's principle, a function $u\in Z$, where~$Z$ is a suitable space, is a solution to~\eqref{eq:CP0} if, and only if, it satisfies the equality
\begin{equation}\label{eq:fixedpoint}
u(t,x) =  u^{\mathrm{lin}}(t,x) + \int_0^t E_1(t-s, x) \ast \, f(u_t(s,x))\, ds\,, \qquad \text{in~$Z$,}
\end{equation}
with $f(u_t(s,x))=|u_t(s,x)|^p$  and
\[ u^{\mathrm{lin}} (t,x) \doteq K_0(t,x) \ast u_0(x) + K_1(t,x) \ast u_1(x) \,, \]
is the solution to the linear Cauchy problem~\eqref{eq:CPlinear}.
The proof of our global existence results is based on the following scheme. We  define an appropriate data function space
\begin{align}
\label{dom:A}
\mathcal A  \doteq  \left(H^{\alpha+\delta}(\R^n)\cap L^{m}(\R^n)\right) \times \left(H^{2\delta}(\R^n)\cap L^{m}(\R^n)\right),
\end{align}
 and an evolution space for solutions
\begin{align}
\label{eq:Xsp}
Z(T) \doteq C([0, T], H^{\alpha}(\R^n)\cap L^{q}(\R^n)) \cap C^1([0, T], L^2(\R^n)\cap L^{q}(\R^n)),
\end{align}
equipped with a  norm relate to the estimates  of solutions to the linear problem \eqref{eq:CPlinear} such that
\begin{equation}\label{eq:ubasic}
\|u^{\mathrm{lin}}\|_{Z} \leq C\,\|(u_0, u_1)\|_{\mathcal A}.
\end{equation}

We define the operator~$F$ such that, for any~$u\in Z$,
\begin{equation}\label{eq:G}\nonumber
Fu(t,x) \doteq \int_0^t E_1(t-s,x)\ast f(u_t(s,x))\, ds\,,
\end{equation}
then we prove the estimates
\begin{align}
\label{eq:well}
\|Fu\|_{Z}
    & \leq C\|u\|_{Z}^p\,, \\
\label{eq:contraction}
\|Fu-Fv\|_{Z}
    & \leq C\|u-v\|_{Z} \bigl(\|u\|_{Z}^{p-1}+\|v\|_{Z}^{p-1}\bigr)\,.
\end{align}
By standard arguments, since $u^{\mathrm{lin}}$ satisfies~\eqref{eq:ubasic} and~$p>1$, from~\eqref{eq:well} it follows that~$u^{\mathrm{lin}}+F$ maps balls of~$Z$ into balls of~$Z$, for small data in~$\mathcal{A}$, and that estimates \eqref{eq:well}-\eqref{eq:contraction} lead to the existence of a unique solution to~\eqref{eq:fixedpoint}, that is, $u=u^{\mathrm{lin}}+Fu$, satisfying~\eqref{eq:ubasic}. We simultaneously gain a local and a global existence result.

\medskip

The information that~$u\in Z$ plays a fundamental role to estimate~$f(u_t(s,\cdot))$ in suitable norms. We will employ the following well-known result (for instance, see \cite{DAE17NA}).
\begin{lemma}\label{lem:integral}
Let~$\kappa\leq 1$. Then it holds
\[
\int_0^t (1+t-s)^{-\kappa}\,\,(1+s)^{-\mu}\,ds \lesssim \left\{ \begin{array}{ll} (1+t)^{-\kappa} \hspace{2cm} \text{if } \mu >1\\ (1+t)^{-\kappa} \log(e+t)\hspace{0,5cm} \text{if } \mu =1.\end{array}\right.
\]
\end{lemma}
\begin{proof}(Theorem \ref{thm:L2p})
We have to prove \eqref{eq:ubasic}, \eqref{eq:well} and \eqref{eq:contraction}, with $\mathcal{A}$ as in \eqref{dom:A} and $Z(T)$ as in \eqref{eq:Xsp} for all $q \geq 2p$, equipped with the norm
\begin{align*}
\nonumber
\|u\|_{Z(T)}
    \doteq& \sup_{t\in[0,T]} \Bigl\{
    (1+t)^{-1+\frac{n}{2\theta}\left(\frac{1}{m}- \frac{1}{2}\right)} \|u(t,\cdot)\|_{L^2}
    +(1+t)^{\frac{1}{2(\alpha-\theta)}\left( n \left(\frac{1}{m}- \frac{1}{2}\right) +\alpha-2\theta\right)} \||D|^{\alpha} u(t,\cdot)\|_{L^2}
    \\
    & \quad + (1+t)^{-1+\frac{n}{2 \theta}\left(\frac{1}{m}- \frac{1}{q}\right)} \|u(t,\cdot)\|_{L^{q}} + (1+t)^{\frac{n}{2\theta}\left(\frac{1}{m}- \frac{1}{2}\right)}\|u_t(t,\cdot)\|_{L^{2}}
    + (1+t)^{\frac{n}{2 \theta}\left(\frac{1}{m}- \frac{1}{q}\right)} \|u_t(t,\cdot)\|_{L^{q}} \Bigr\}.
\end{align*}
Thanks to Theorem \ref{teo>} and Remark \ref{usefulrem0}, $u^{\mathrm{lin}} \in Z(T)$ and it satisfies \eqref{eq:ubasic}.

Let us prove \eqref{eq:well}. We omit the proof of~\eqref{eq:contraction}, since it is analogous to the proof of~\eqref{eq:well}.

Let $u\in Z(T)$.
Using  Theorem \ref{teo>} with $\eta=m$ and Remark \ref{usefulrem0}, for $j=0, 1$ we have
\begin{eqnarray*}
\| \partial_t^j E_1(t-s,\cdot)\ast f(u_t(s,\cdot))\|_{L^2} &\lesssim& (1+t-s)^{1-j-\frac{n}{2\theta}\left(\frac{1}{m}- \frac{1}{2}\right)}
\|f(u_t(s,\cdot))\|_{L^2\cap L^m},
\\
\| \partial_t^j E_1(t-s,\cdot)\ast f(u_t(s,\cdot))\|_{L^{q}} &\lesssim& (1+t-s)^{1-j-\frac{n}{2\theta}\left(\frac{1}{m}-\frac{1}{q}\right)} \|f(u_t(s,\cdot))\|_{L^2\cap L^m},
\\
\| |D|^\alpha E_1(t-s,\cdot)\ast f(u_t(s,\cdot))\|_{L^2} &\lesssim& (1+t-s)^{-\frac{1}{2(\alpha-\theta)}\left( n \left(\frac{1}{m}- \frac{1}{2}\right) +\alpha-2\theta\right)}
\|f(u_t(s,\cdot))\|_{L^2\cap L^m}.
\end{eqnarray*}
Condition \eqref{m0} implies that $m\left(1+ \frac{2m\theta}{n} \right)\geq 2$, hence, for $\kappa\geq m$ and $ 1+\frac{2m\theta}{n}<p\leq \frac{q}2 $,
by interpolation we may estimate
\[ \|f(u_t(s,\cdot))\|_{L^\kappa}= \|u_t(s, \cdot)\|_{L^{\kappa p}}^p  \lesssim (1+s)^{-\frac{np}{2\theta}\left(\frac{1}{m}-\frac1{\kappa p}\right)} \|u\|^p_{Z(T)},\]
 and Lemma \ref{lem:integral} implies
 \begin{eqnarray*}
 \| \partial_t^jFu(t, \cdot)\|_{L^2}&\lesssim& \| u\|_{Z(T)}^p\int_0^t (1+t-s)^{1-j-\frac{n}{2\theta}\left(\frac{1}{m}- \frac{1}{2}\right)}(1+s)^{-\frac{np}{2\theta}\left(\frac{1}{m}-\frac1{m p}\right)}ds\\
 &\lesssim&
(1+t)^{1-j-\frac{n}{2\theta}\left(\frac{1}{m}- \frac{1}{2}\right)}\| u\|_{Z(T)}^p,
\end{eqnarray*}
\begin{eqnarray*} \| \partial_t^jFu(t, \cdot)\|_{L^{q}} &\lesssim& \| u\|_{Z(T)}^p\int_0^t (1+t-s)^{1-j-\frac{n}{2\theta}\left(\frac{1}{m}-\frac{1}{q}\right)}(1+s)^{-\frac{np}{2\theta}\left(\frac{1}{m}-\frac1{m p}\right)}ds \\ &\lesssim&
(1+t)^{1-j-\frac{n}{2\theta}\left(\frac{1}{m}-\frac{1}{q}\right)}\| u\|_{Z(T)}^p,
\end{eqnarray*}
and
\begin{eqnarray*}
 \| |D|^\alpha Fu(t, \cdot)\|_{L^2}&\lesssim& \| u\|_{Z(T)}^p\int_0^t  (1+t-s)^{-\frac{1}{2(\alpha-\theta)}\left( n \left(\frac{1}{m}- \frac{1}{2}\right) +\alpha-2\theta\right)}(1+s)^{-\frac{np}{2\theta}\left(\frac{1}{m}-\frac1{m p}\right)}ds\\
 &\lesssim&
 (1+t)^{-\frac{1}{2(\alpha-\theta)}\left( n \left(\frac{1}{m}- \frac{1}{2}\right) +\alpha-2\theta\right)}\| u\|_{Z(T)}^p,
\end{eqnarray*}
thanks again to $p> 1+\frac{2m\theta}{n} $ and
 $q \leq \frac{n\,m}{(n-2m\theta)_+}$.
\end{proof}
\begin{proof}(Theorem \ref{thm:theta0})
 We have to prove \eqref{eq:ubasic}, \eqref{eq:well} and \eqref{eq:contraction}, with $\mathcal{A}$ as in \eqref{dom:A} for $m=2$ and $Z(T)$ as in \eqref{eq:Xsp} for all $q \geq 2p$,
equipped with the norm
\begin{align*}
\nonumber
\|u\|_{Z(T)}
    \doteq \sup_{t\in[0,T]} & \Bigl\{
     \|u(t,\cdot)\|_{L^2}
    +(1+t)^{\frac{1}{2}} \||D|^{\alpha} u(t,\cdot)\|_{L^2}+ (1+t)^{\min\left\{\frac{n}{2\alpha}\left(\frac12- \frac{1}{q}\right), \frac{\alpha+\delta}{2\delta}- \frac{n}{2\delta}\left(\frac12- \frac{1}{q}\right)\right\}} \|u(t,\cdot)\|_{L^{q}}
    \\
     & \qquad  + (1+t)\|u_t(t,\cdot)\|_{L^{2}}
    + (1+t)^{1-\frac{n}{2\delta}\left(\frac12-\frac1q\right)} \|u_t(t,\cdot)\|_{L^{q}} \Bigr\}.
\end{align*}
We only discuss the proof  of \eqref{eq:well}.  Let $u\in Z(T)$.
Using  Theorem \ref{teo>} with $\eta=2$, for $j=0, 1$ we have
\begin{eqnarray*}
\| \partial_t^j E_1(t-s,\cdot)\ast f(u_t(s,\cdot))\|_{L^2} &\lesssim& (1+t-s)^{-j}
\|f(u_t(s,\cdot))\|_{L^2},
\\
\| |D|^\alpha E_1(t-s,\cdot)\ast f(u_t(s,\cdot))\|_{L^2}&\lesssim& (1+t-s)^{-\frac{1}{2}}
\|f(u_t(s,\cdot))\|_{L^2},
\\
\|  E_1(t-s,\cdot)\ast f(u_t(s,\cdot))\|_{L^{q}} &\lesssim& (1+t-s)^{-\min\left\{\frac{n}{2\alpha}\left(\frac12- \frac{1}{q}\right), \frac{\alpha+\delta}{2\delta}- \frac{n}{2\delta}\left(\frac12- \frac{1}{q}\right)\right\}} \|f(u_t(s,\cdot))\|_{L^2},
\\	
\| \partial_t E_1(t-s,\cdot)\ast f(u_t(s,\cdot))\|_{L^{q}} &\lesssim & (1+t)^{\frac{n}{2\delta}\left(\frac12-\frac1q\right)-1}\|f(u_t(s,\cdot))\|_{L^2}.
\end{eqnarray*}
By interpolation and  the definition of $\| \cdot \|_{Z(T)}$ we have
\[ \|f(u_t(s,\cdot))\|_{L^2}= \|u_t(s, \cdot)\|_{L^{2p}}^p  \lesssim (1+s)^{\frac{n}{4\delta}\left(p-1\right)-p} \|u\|^p_{Z(T)}.\]
 Lemma \ref{lem:integral} implies
 \begin{eqnarray*}
 \| \partial_t^jFu(t, \cdot)\|_{L^2}&\lesssim& \| u\|_{Z(T)}^p\int_0^t (1+t-s)^{-j}(1+s)^{\frac{n}{4\delta}\left(p-1\right)-p}ds\\
 &\lesssim&
(1+t)^{-j}\| u\|_{Z(T)}^p,
\end{eqnarray*}
\begin{eqnarray*}
	\| |D|^\alpha Fu(t, \cdot)\|_{L^2}&\lesssim& \| u\|_{Z(T)}^p\int_0^t  (1+t-s)^{-\frac{1}{2}}(1+s)^{\frac{n}{4\delta}\left(p-1\right)-p}ds\\
	&\lesssim&
	(1+t)^{-\frac12}\| u\|_{Z(T)}^p,
\end{eqnarray*}
 \begin{eqnarray*}
 \| Fu(t, \cdot)\|_{L^{q}}&\lesssim& \| u\|_{Z(T)}^p\int_0^t (1+t-s)^{-\min\left\{\frac{n}{2\alpha}\left(\frac12- \frac{1}{q}\right), \frac{\alpha+\delta}{2\delta}- \frac{n}{2\delta}\left(\frac12- \frac{1}{q}\right)\right\}}(1+s)^{\frac{n}{4\delta}\left(p-1\right)-p}ds\\
 &\lesssim&
(1+t)^{-\min\left\{\frac{n}{2\alpha}\left(\frac12- \frac{1}{q}\right), \frac{\alpha+\delta}{2\delta}- \frac{n}{2\delta}\left(\frac12- \frac{1}{q}\right)\right\}}\| u\|_{Z(T)}^p,
\end{eqnarray*}
and
 \begin{eqnarray*}
 \| \partial_t Fu(t, \cdot)\|_{L^{q}}&\lesssim& \| u\|_{Z(T)}^p\int_0^t (1+t-s)^{\frac{n}{2\delta}\left(\frac12-\frac1q\right)-1}(1+s)^{\frac{n}{4\delta}\left(p-1\right)-p}ds\\
 &\lesssim&
(1+t)^{\frac{n}{2\delta}\left(\frac12-\frac1q\right)-1}\| u\|_{Z(T)}^p,
\end{eqnarray*}
thanks to  $p>1$ and $n< 4\delta$.
\end{proof}
\begin{proof}(Theorem \ref{thm:sp})
We have to prove \eqref{eq:ubasic}, \eqref{eq:well} and \eqref{eq:contraction}, with $\mathcal{A}$ as in \eqref{dom:A} and $Z(T)$ as in \eqref{eq:Xsp} for all $q \geq 2p$, equipped with the norm
		\begin{align}
		\nonumber
		&\|u\|_{Z(T)}
		\doteq  \sup_{t\in[0,T]} \Bigl\{
		(1+t)^{-1+\frac{n}{2\theta}\left(\frac{1}{m}- \frac{1}{2}\right)} \|u(t,\cdot)\|_{L^2}
		+(1+t)^{\frac{1}{2(\alpha-\theta)}\left( n \left(\frac{1}{m}- \frac{1}{2}\right) +\alpha-2\theta\right)} h(t) \||D|^{\alpha} u(t,\cdot)\|_{L^2}
		\\
		\label{eq:XkspLm}
		&\qquad + (1+t)^{-1+\frac{n}{2 \theta}\left(\frac{1}{m}- \frac{1}{q}\right)} \|u(t,\cdot)\|_{L^{q}} + (1+t)^{\frac{n}{2\theta}\left(\frac{1}{m}- \frac{1}{2}\right)}\|u_t(t,\cdot)\|_{L^{2}}
		+ (1+t)^{\frac{n}{2 \theta}\left(\frac{1}{m}- \frac{1}{q}\right)} \|u_t(t,\cdot)\|_{L^{q}} \Bigr\},
		\end{align}
where $h(t)=(\log(e+t))^{-1}$ if $\alpha>2\theta$ and $h(t)\equiv 1$ if $\alpha=2\theta$.

		Thanks to Theorem \ref{teo>} and Remark \ref{usefulrem0}, $u^{\mathrm{lin}} \in Z(T)$ and it satisfies \eqref{eq:ubasic}.

		Let us prove \eqref{eq:well}. Let $u\in Z(T)$. To consider the case $p= p_c\doteq 1+\frac{2m\theta}{n}$ we have to change  the argument done in Theorem \ref{thm:L2p}.
Condition \eqref{m0} implies that $m\left(1+ \frac{2m\theta}{n} \right)\geq 2$, hence, for $\kappa\geq m$ and $ 1+\frac{2m\theta}{n}=p\leq \frac{q}2 $, by interpolation and \eqref{eq:XkspLm} we have
	\begin{equation}\label{normtheo2.4}		
		\|f(u_t(s,\cdot))\|_{L^\kappa}= \|u_t(s, \cdot)\|_{L^{\kappa p}}^p  \lesssim (1+s)^{-\frac{np}{2\theta}\left(\frac{1}{m}-\frac1{\kappa p}\right)} \|u\|^p_{Z(T)}.
	\end{equation}
		Applying Theorem \ref{teo>} with $1<\eta<m$ such that $\eta\left(1+ \frac{2m\theta}{n} \right)>2$ and $n\left(\frac1{\eta}- \frac{1}{q}\right)\leq 2\theta$, we may estimate
		\begin{eqnarray*}
			\|\partial_t^jFu(t, \cdot)\|_{L^2}&\lesssim& \int_0^t (1+t-s)^{1-j-\frac{n}{2\theta}\left(\frac1{\eta}-\frac12\right)} \|f(u_t(s,\cdot))\|_{L^{2}\cap L^\eta}ds\\
			&\lesssim& \int_0^t (1+t-s)^{1-j-\frac{n}{2\theta}\left(\frac1{\eta}-\frac12\right)} \|u_t(s,\cdot)\|_{L^{2 p_c}\cap L^{\eta p_c}}^{p_c}ds, \qquad j=0,1.
		\end{eqnarray*}
		Using \eqref{normtheo2.4} we conclude
		\[
		\|\partial_t^jFu(t, \cdot)\|_{L^2}\lesssim \| u\|_{Z(T)}^{p_c} \int_0^t (1+t-s)^{1-j-\frac{n}{2\theta}\left(\frac1{\eta}-\frac12\right)} (1+s)^{-\frac{n}{2\theta}\left(\frac{p_c}{m}-\frac1{\eta}\right)}ds.
		\]
		Now we split the integration  interval into $[0, t/2]$ and $[t/2, t]$:
		\begin{eqnarray*}
			\int_0^{\frac{t}2} (1+t-s)^{1-j-\frac{n}{2\theta}\left(\frac1{\eta}-\frac12\right)} (1+s)^{-\frac{n}{2\theta}\left(\frac{p_c}{m}-\frac1{\eta}\right)}ds &\lesssim&
			(1+t)^{1-j-\frac{n}{2\theta}\left(\frac1{\eta}-\frac12\right)}\int_0^{\frac{t}2} (1+s)^{-\frac{n}{2\theta}\left(\frac{p_c}{m}-\frac1{\eta}\right)}ds\\
			&\lesssim&(1+t)^{1-j-\frac{n}{2\theta}\left(\frac{1}{m}-\frac12\right)}.
		\end{eqnarray*}
		On the other hand,  we conclude
		\begin{eqnarray*}
			\int_{\frac{t}2}^t (1+t-s)^{1-j-\frac{n}{2\theta}\left(\frac1{\eta}-\frac12\right)} (1+s)^{-\frac{n}{2\theta}\left(\frac{p_c}{m}-\frac1{\eta}\right)}ds &\lesssim&
			(1+t)^{-\frac{n}{2\theta}\left(\frac{p_c}{m}-\frac1{\eta}\right)}\int_{\frac{t}2}^t (1+t-s)^{1-j-\frac{n}{2\theta}\left(\frac1{\eta}-\frac12\right)}ds\\
			&\lesssim&(1+t)^{1-j-\frac{n}{2\theta}\left(\frac{1}{m}-\frac12\right)}.
		\end{eqnarray*}
		Hence
		\[ \|\partial_t^jFu(t, \cdot)\|_{L^2}\lesssim(1+t)^{1-j-\frac{n}{2\theta}\left(\frac{1}{m}-\frac12\right)} \| u\|_{Z(T)}^{p_c}, \qquad j=0,1.
		\]
		Similarly,
		\[\| \partial_t^jFu(t, \cdot)\|_{L^{q}}\lesssim
		(1+t)^{1-j-\frac{n}{2 \theta}\left( \frac{1}{m} - \frac{1}{q} \right)}\| u\|_{Z(T)}^{p_c}, \qquad j=0,1 \]
		and
		\[\||D|^{\alpha} Fu(t, \cdot)\|_{L^2}
		\lesssim  (1+t)^{-\frac{1}{2(\alpha-\theta)}\left( n \left(\frac{1}{m}- \frac{1}{2}\right) +\alpha-2\theta\right)}\| u\|_{Z(T)}^{p_c}, \qquad \text{if } \alpha = 2 \theta.\]
		However, compared with the estimates for solutions to the linear problem, in the case $p=p_c$ we have a logarithm loss of decay to the term $\||D|^{\alpha} Fu(t, \cdot)\|_{L^2}$ if $\alpha > 2 \theta$. Applying Theorem \ref{teo>} and Lemma \ref{lem:integral} with $\mu=1$ we have
		\begin{eqnarray*}
			\||D|^{\alpha} Fu(t, \cdot)\|_{L^2}&\lesssim& \int_0^t (1+t-s)^{-\frac{1}{2(\alpha-\theta)}\left( n \left(\frac{1}{m}- \frac{1}{2}\right) +\alpha-2\theta\right)} \|f(u_t(s,\cdot))\|_{L^{2}\cap L^m}ds\\
			&\lesssim& \| u\|_{Z(T)}^{p_c}\int_0^t (1+t-s)^{-\frac{1}{2(\alpha-\theta)}\left( n \left(\frac{1}{m}- \frac{1}{2}\right) +\alpha-2\theta\right)} (1+s)^{-\frac{n}{2m\theta}\left(p_c-1\right)}ds\\
			&\lesssim&  (1+t)^{-\frac{1}{2(\alpha-\theta)}\left( n \left(\frac{1}{m}- \frac{1}{2}\right) +\alpha-2\theta\right)}\log(e+t)\| u\|_{Z(T)}^{p_c}.
		\end{eqnarray*}
\end{proof}
\section*{acknowledgements}
The first author have been partially supported by Funda\c{c}\~{a}o de Amparo \`{a} Pesquisa do Estado de S\~{a}o Paulo (FAPESP),
grant number 2017/19497-3. The second author has been partially supported by Conselho Nacional de Desenvolvimento Cient\'ifico e
Tecnol\'ogico - CNPq, Proc.  308868/2015-3 and 314398/2018-0.

\end{document}